\documentclass{article}%
\usepackage{amsfonts}
\usepackage{amsmath}%
\usepackage{amssymb}%
\usepackage{graphicx}
\providecommand{\U}[1]{\protect\rule{.1in}{.1in}}
\newtheorem{theorem}{Theorem}

\newtheorem{corollary}[theorem]{Corollary}

\newtheorem{lemma}[theorem]{Lemma}

\newenvironment{proof}[1][Proof]{\noindent\textbf{#1.} }{\ \rule{0.5em}{0.5em}}
\begin{document}

\title{The Automorphism Group of a Self Dual Binary [72,36,16] Code Does Not Contain
$Z_{4}$}
\author{Vassil Yorgov \thanks{Department of Mathematics and Computer Science,
Fayetteville State University, NC.} \ and Daniel Yorgov~\thanks{Department of
Mathematical \& Statistical Sciences, University of Colorado Denver.}}
\date{}
\maketitle

\begin{abstract}
It has been proven in a series of works that the order of the automorphism
group of a binary [72,36,16] code does not exceed five. We obtain a
parametrization of all self-dual binary codes of length 72 with automorphism
of order 4 which can be extremal. We use extensive computations in MAGMA and
on a supercomputer to show that an extremal binary code of length 72 does not
have an element of order 4.

\end{abstract}


\section{Introduction}

A binary self-dual code of length $n$ is doubly-even if the weight of every
code vector is a multiple of four. An upper bound $d\leq4\lfloor
n/24\rfloor+4$ for the minimum weight $d$ of such code is given in
\cite{MallowsSl}. Self-dual codes achieving this bound are called extremal.
The extremal codes of length divisible by 24 are of particular interest. There
are unique extremal codes of length 24 and 48 \cite{Pless1968},
\cite{HoughtonLam}. It is not known if an extremal code of length 72 exists
\cite{Sloane73}.

Let $C$ be a binary extremal self-dual \ doubly-even code of length 72. One
possible way of looking for such code is by assuming that $C$ has nontrivial
automorphisms. The automorphism group of $C$ is studied in a series of works.
The possible prime divisors of the group are determined in \cite{ConwayPless}.
The prime divisors bigger than 5 are eliminated in \cite{Pless1982},
\cite{PlessThom}, \cite{HuffYor} , and \cite{FeulnerNebe}. The automorphism
group of $C$ and its possible order is studied in \cite{Bouy2004},
\cite{Yor2006}, \cite{BouyBriWill}, \cite{Yankov2012}, \cite{Nebe2012},
\cite{Borello2012} , \cite{BoreVoltaNebe}, and \cite{Borello2014}. As a result
it is known that the order of the automorphism group of $C$ is at most 5.

In this note we search for a code $C$ which allows the cyclic group of degree
4, $Z_{4},$ as an automorphism group. The subcode of $C$ fixed by an
automorphism of order 2 corresponds to a self dual [36,18,8] binary code. All
such codes are completely classified \cite{AguilarGaborit}. Up to equivalence
their number is 41. We prove that only three of them correspond to $C$, namely
the codes $C_{4},$ $C_{12},$ and $C_{19}$ from the Munemasa's online database
of binary self-dual codes \cite{Munemasa}. Using some resent results on the
module structure of $C$ \cite{Nebe2012}, we obtain that a generator matrix of
$C$ depends on 72 binary parameters. Each of these 3(2\symbol{94}72) matrices
generates a doubly even self-dual code of length 72 with cyclic automorphism
group of order 4. We further decrease the search space by using affine
transformations which preserve the fixed subcode that determines $C_{4},$
$C_{12},$ and $C_{19},$ correspondingly. These transformations act on a
certain 27 dimensional subcode of $C$ which matrix depends only on the first
36 parameters. Using computer algebra system Magma \cite{BosmaMagma} and a
desktop computer we split these subcodes into orbits and compute the minimum
weight of a representative of each orbit. The total number of orbits of weight
16 subcodes is 1558954. Each of them determines 2\symbol{94}36 generator
matrices for $C$ which is still a formidable search space. These
1558954(2\symbol{94}36) binary matrices were traversed and checked on the
Janus supercomputer at the University of Colorado. Each matrix generated a
vector of weight less than 16. Thus we have the following result.

\begin{theorem}
\label{main} The automorphism group of a binary self-dual [72,36,16] code does
not have an element of order four.
\end{theorem}

\section{Parametrization of $C$}

In this section $C\leq\mathbb{F}_{2}^{72}$ is a self-dual binary code of
length 72 with minimum weight 16 and automorphism $g$ of order 4. As $g^{2}$
has order 2, it does not have fixed points \cite{Bouy2002} . Hence, $g$ is
free of fixed points and we may assume that \
\begin{equation}
g=(1,2,3,4)(5,6,7,8)\cdots(69,70,71,72).\label{eqn_g}%
\end{equation}
It is known \cite{Nebe2012} that $C$ is a free module over the group ring
$R=\mathbb{F}_{2}\left\langle g\right\rangle .$ As the dimension of $R$ as
vector space over $\mathbb{F}_{2}$\ \ is 4, the free $R$ -module $C$ has rank 9.

Let
\[
v=(v_{1,0},v_{1,1},v_{1,2},v_{1,3},v_{2,0},v_{2,1},v_{2,2},v_{2,3}%
,\ldots,v_{18,0},v_{18,1},v_{18,2},v_{18,3})
\]
be a vector from $\mathbb{F}_{2}^{72}.$ We define the maps $\mu:\mathbb{F}%
_{2}^{72}\rightarrow R^{18}$ and $\mu^{\prime}:\mathbb{F}_{2}^{72}\rightarrow
R^{18}$\ by
\begin{align*}
\mu(v)  & =\left(  \sum\limits_{i=0}^{3}v_{1,i}g^{i},\sum\limits_{i=0}%
^{3}v_{2,i}g^{i},\ldots,\sum\limits_{i=0}^{3}v_{18,i}g^{i}\right)  ,\\
\mu^{\prime}(v)  & =\left(  \sum\limits_{i=0}^{3}v_{1,i}g^{-i},\sum
\limits_{i=0}^{3}v_{2,i}g^{-i},\ldots,\sum\limits_{i=0}^{3}v_{18,i}%
g^{-i}\right)  .
\end{align*}
Thus $\mu(C)$ is a code of length 18 over the ring $R$ and has rank 9.

A version of the following lemma for automorphisms of odd prime orders is
proved in \cite{Yor1983}.

\begin{lemma}
\label{Lmu} Let $u$ and $v$ be vectors from $\mathbb{F}_{2}^{72}.$ The vector
$u$ is orthogonal to the vectors $g^{i}v$ for $i=0,1,2,3$ if and only if
$\mu(u)$ and $\mu^{\prime}(v)$ are orthogonal with respect to the usual inner
product in $R^{18}.$
\end{lemma}

Let $h=1+g.$ Then $h^{2}=1+g^{2},$ $\ h^{3}=1+g+g^{2}+g^{3},$ $h^{4}%
=1+g^{4}=0,$ and $1,~h,~h^{2},~h^{3}$ is a basis of $R$ over $\mathbb{F}_{2}.$
Thus, $R=\mathbb{F}_{2}+\mathbb{F}_{2}h+\mathbb{F}_{2}h^{2}+\mathbb{F}%
_{2}h^{3}.$ The next obvious lemma is included for convenience in referencing.

\begin{lemma}
\label{Lh} The $\mathbb{F}_{2}-$ linear transformation on the ring $R$ defined
by $g^{i}\longmapsto g^{-i},$ for $i=1,\ 2,~3,$ maps the basis elements
$1,~h,~h^{2},~h^{3}$ to $1,~h+h^{2}+h^{3},~h^{2},~h^{3},~$correspondingly.
\end{lemma}

In the next theorem we obtain further restrictions on the $R$-code $\mu(C).$

\begin{theorem}
\label{Structure4} Let $C$ be a self-dual doubly-even binary code of length 72
with minimum weight 16 and automorphism $g$ of order 4 given in (\ref{eqn_g}).
Let $h=1+g.$ Up to a column permutation, $\mu(C)$ is a code of length 18 over
the ring $R$ with a generator matrix $[I+B_{1}h+B_{2}h^{2}+B_{3}h^{3}~|~A]$
where $I$ is the identity matrix and$~B_{1},\ B_{2},\ B_{3},$ and $A$ are
binary square matrices of order 9 satisfying the following
requirements:\newline(i) $A$ is orthogonal ($A^{T}A=I$ where $A^{T}$ is the
transpose of $A),$ \newline(ii) $B_{1}$ is symmetric $(B_{1}^{T}=B_{1}),$
\newline(iii) $B_{2}+B_{2}^{T}=B_{1}^{2}+B_{1},$\newline(iv) $B_{3}+B_{3}%
^{T}=B_{2}B_{1}+B_{1}B_{2}+B_{1}^{3}+B_{1}.$ \newline If $B_{1},\ B_{2}%
,\ B_{3},$ and $A$ satisfy the above four conditions, then the $\mu$-preimage
of an $R$-code with a generator matrix $[I+B_{1}h+B_{2}h^{2}+B_{3}h^{3}~|~A]$
\ is a binary self-dual code of length 72.
\end{theorem}

\begin{proof}
As we know $\mu(C)$ is a free $R$ module of rank 9 \cite{Nebe2012}. Since
$h^{4}=0,$ $1+\mathbb{F}h+\mathbb{F}h^{2}+\mathbb{F}h^{3}$ is the set of all
units of $R.$ It follows that, up to a permutation of columns, the code
$\mu(C)$ has a generator matrix
\begin{equation}
\lbrack I~|~A+A_{1}h+A_{2}h^{2}+A_{3}h^{3}]\label{mat1}%
\end{equation}
where $I$ is the identity matrix of order 9 and $A,~A_{1},~A_{2,}~A_{3}$ are
binary square matrices of order 9. Lemma \ref{Lmu} and Lemma \ref{Lh} imply
\[
I+(A+A_{1}h+A_{2}h^{2}+A_{3}h^{3})(A^{T}+A_{1}^{T}(h+h^{2}+h^{3})+A_{2}%
^{T}h^{2}+A_{3}^{T}h^{3})=0.
\]
As $1,~h,~h^{2},~h^{3}$ are linearly independent over $\mathbb{F}_{2},$\ \ we
obtain $I+AA^{T}=0.$ Thus $A$ is orthogonal. \newline Using this we can
replace (\ref{mat1}) with%
\begin{equation}
\lbrack I+B_{1}h+B_{2}h^{2}+B_{3}h^{3}~|~A].\label{gmain}%
\end{equation}
Additional applications of Lemma \ref{Lmu} and Lemma \ref{Lh} on the matrix
(\ref{gmain}) give parts (ii), (iii), and (iv). The proof of the second part
of the Theorem is straightforward.
\end{proof}

Since the matrix $B_{2}+B_{2}^{T}$ from condition (iii) has zero diagonal, we
obtain the following corollary.

\begin{corollary}
\label{diag} The matrix $B_{1}$ from Theorem \ref{Structure4} is such that
$B_{1}^{2}+B_{1}$ has zero diagonal.
\end{corollary}

Any $R$ - code with a generator matrix satisfying the requirements of Theorem
\ref{Structure4} corresponds under the map $\mu^{-1}$ to a binary self-dual
code of length 72 which is not necessarily doubly-even. Since an extremal code
of length 72 is doubly-even, the search space can be restricted. It is done in
the next theorem which provides a necessary and sufficient condition for a
code to be doubly-even.

We use the following well known Lemma (see \cite{HuffPless03}, page 8) in the
proof of Theorem \ref{Th_de}.

\begin{lemma}
\label{mod4} If $u$ and $v$ are binary vectors of the same length, then
\[
wt(u+v)\equiv wt(u)+wt(v)+2(u,v)~\operatorname{mod}4
\]
where $(u,v)$ is the inner product.
\end{lemma}

\begin{theorem}
\label{Th_de} Let $D$ be an $R$ - code with a generator matrix $[I+B_{1}%
h+B_{2}h^{2}+B_{3}h^{3}~|~A]$ satisfying the requirements (i), (ii), (iii),
and (iv) of Theorem \ref{Structure4}. The corresponding binary code $\mu
^{-1}(D)$ of length 72 is doubly-even if and only if
\[
B_{3}[i,i]\equiv\frac{1+wt(A[i])}{2}+B_{2}[i,i]+\sum_{j=1}^{9}(B_{1}%
[i,j]+1)B_{2}[i,j]~\operatorname{mod}2
\]
for $i=1,2,\ldots,9$ where $A[i]$ is row $i$ of the matrix $A$ and
$B_{2}[i,j]$ is the entry in row $i$ column $j$ of the matrix $B_{2}.$
\end{theorem}

\begin{proof}
Since $h^{4}=1,$ the matrices $[Ih+B_{1}h^{2}+B_{2}h^{3}~|~Ah],$
$[Ih^{2}+B_{1}h^{3}~|~Ah^{2}],$ and $[Ih^{3}~|~Ah^{3}]$ generate subcodes of
$D.$ Replacing $h$ with $1+g$ and collecting the terms with respect to the
powers of $g,$ we obtain the matrix%

\begin{equation}%
\begin{bmatrix}
I+B_{1}+B_{2}+B_{3} & B_{1}+B_{3} & B_{2}+B_{3} & B_{3} & A & 0 & 0 & 0\\
I+B_{1}+B_{2} & I+B_{2} & B_{1}+B_{2} & B_{2} & A & A & 0 & 0\\
I+B_{1} & B_{1} & I+B_{1} & B_{1} & A & 0 & A & 0\\
I & I & I & I & A & A & A & A
\end{bmatrix}
\label{mm-1}%
\end{equation}
which generates a code equivalent to $\mu^{-1}(D).$ Hence the code is
self-dual by Theorem \ref{Structure4}.

Let $w_{i}$ be the weight of row $i$ of matrix (\ref{mm-1}). Thus $w_{i}$ is
even for $i=1,2,\ldots,36.$ The code is doubly-even if and only if $w_{i}$ is
a multiple of 4.\newline For $i=1,2,\ldots,9$ we have
\begin{align*}
w_{i}  & =wt(I[i]+B_{1}[i]+B_{2}[i]+B_{3}[i])+wt(B_{1}[i]+B_{3}[i])+\\
& wt\left(  B_{2}[i]+B_{3}[i]\right)  +wt\left(  B_{3}[i]\right)  +wt(A[i]).
\end{align*}
\newline Applying Lemma \ref{mod4} several times, we obtain
\begin{align*}
w_{i}  & \equiv wt(I[i])+wt(A[i])+2wt(B_{1}[i])+2wt(B_{2}[i])+\\
& 2(I[i],B_{1}[i])+2(I[i],B_{2}[i])+2(I[i],B_{3}[i])+2(B_{1}[i],B_{2}%
[i])~\operatorname{mod}4.
\end{align*}
Since $A$ is orthogonal, $wt(A[i])$ is odd and
\begin{align*}
\frac{w_{i}}{2}  & \equiv\frac{1+wt(A[i])}{2}+wt(B_{1}[i])+wt(B_{2}[i])+\\
& B_{1}[i,i]+B_{2}[i,i]+B_{3}[i,i]+(B_{1}[i],B_{2}[i])~\operatorname{mod}2
\end{align*}
\newline As $wt(B_{1}[i])\equiv$ $B_{1}B_{1}^{T}[i,i]$ $~\operatorname{mod}2$
and $B_{1}$ is symmetric, we have
\[
B_{1}[i,i]+wt(B_{1}[i])\equiv B_{1}[i,i]+B_{1}^{2}[i,i]~\operatorname{mod}2.
\]
Now Corollary (\ref{diag}) gives $B_{1}[i,i]+wt(B_{1}[i])\equiv
0~\operatorname{mod}2.$ Hence, $w_{i}$ is a multiple of 4 if and only if
\[
B_{3}[i,i]\equiv\frac{1+wt(A[i])}{2}+B_{2}[i,i]+wt(B_{2}[i])+(B_{1}%
[i],B_{2}[i])~\operatorname{mod}2.
\]
\newline For $i=10,~11,~\ldots,36$\ we check similarly that $w_{i}$ is always
a multiple of 4.
\end{proof}

Now we determine the possible matrices $B_{1}$ and $A$ for the code $C.$ We
use the automorphism $g^{2}=(1,3)(2,4)\cdots(71,73)(72,74)$ to define two
mappings. Let $C(g^{2})=\left\{  c\in C~|~g^{2}(c)=c\right\}  $ be the fixed
code of $g^{2}.$ The first mapping is $\pi:C(g^{2})\rightarrow\mathbb{F}%
_{2}^{36}$ defined by
\begin{equation}
\left(  c_{1},c_{2},c_{1},c_{2},\ldots,c_{35,}c_{36},c_{35,}c_{36}\right)
\longmapsto\left(  c_{1},c_{2},\ldots,c_{36}\right)  .\label{pi}%
\end{equation}
The second mapping is $\Phi:C\rightarrow\mathbb{F}_{2}^{36}$ defined by
\[
\left(  c_{1},c_{2},\ldots,c_{72}\right)  \longmapsto\left(  c_{1}+c_{3}%
,c_{2}+c_{4},\ldots,c_{70}+c_{72}\right)  .
\]
It is known \cite{Nebe2012} that $\pi(C(g^{2}))=\Phi(C)$ is a self-dual
[36,18,8] binary code. An application of $\Phi$ is the same as identifying
coordinate positions 1 and 3, 2 and 4, and so on. This identification makes
$g^{2}$ trivial and $\Phi(C)$ becomes a module over the quotient ring
$R/\left\langle g^{2}\right\rangle \cong\mathbb{F}_{2}+\overline{h}%
\mathbb{F}_{2}$ where $\overline{h}$ is the coset $h\left\langle
h^{2}\right\rangle $ and $\overline{h}^{2}=0.$ Thus, $\Phi(C)$ is generated by
the matrix $[I+B_{1}\overline{h}~|~A]$ over the ring $\mathbb{F}_{2}%
+\overline{h}\mathbb{F}_{2}$ and has automorphism $\overline{g}%
=(1,2)(3,4)\cdots(34,36)$ as a binary code. There are 41 inequivalent
self-dual [36,18,8] binary codes \cite{AguilarGaborit}. For each of these 41
codes we find the conjugacy classes of automorphisms of order 2 without fixed
points and select a representative. For each pair of code and orbit
representative we compute the matrix $B_{1}.$ Only for three of the pairs the
requirement of Corollary \ref{diag} are met. They come from the codes $C_{4},$
$C_{12},$ and $C_{19}$ from the Munemasa's online database of binary self-dual
codes \cite{Munemasa}. We reorder the coordinates of the three codes in such
way that $\overline{g}=(1,2)(3,4)\cdots(34,36)$ is the automorphism of order 2
in each of the three pairs. This way we obtain the following lemma.

\begin{lemma}
\label{B1A} Let $C$ be a self-dual doubly-even binary code of length 72 with
minimum weight 16 and automorphism $g$ of order 4 given in (\ref{eqn_g}). Up
to equivalence, the matrices $B_{1}$ and $A$ from Theorem \ref{Structure4} can
be selected as follows:\newline(i) \newline$B_{1}^{(1)}=%
\begin{bmatrix}
0 & 1 & 0 & 0 & 0 & 0 & 0 & 1 & 0\\
1 & 0 & 0 & 0 & 0 & 0 & 0 & 1 & 0\\
0 & 0 & 0 & 0 & 0 & 1 & 1 & 0 & 0\\
0 & 0 & 0 & 0 & 0 & 1 & 0 & 1 & 0\\
0 & 0 & 0 & 0 & 0 & 0 & 0 & 1 & 1\\
0 & 0 & 1 & 1 & 0 & 0 & 1 & 1 & 0\\
0 & 0 & 1 & 0 & 0 & 1 & 0 & 1 & 1\\
1 & 1 & 0 & 1 & 1 & 1 & 1 & 0 & 0\\
0 & 0 & 0 & 0 & 1 & 0 & 1 & 0 & 0
\end{bmatrix}
,$ ~$A^{(2)}=%
\begin{bmatrix}
0 & 1 & 0 & 0 & 1 & 0 & 0 & 0 & 1\\
0 & 1 & 1 & 1 & 0 & 1 & 1 & 1 & 1\\
1 & 0 & 1 & 0 & 0 & 0 & 1 & 0 & 0\\
1 & 1 & 1 & 0 & 0 & 0 & 0 & 1 & 1\\
0 & 0 & 1 & 1 & 1 & 0 & 1 & 0 & 1\\
0 & 1 & 0 & 1 & 1 & 1 & 0 & 1 & 0\\
1 & 1 & 0 & 0 & 0 & 0 & 1 & 1 & 1\\
0 & 0 & 0 & 0 & 1 & 0 & 0 & 1 & 1\\
0 & 0 & 1 & 0 & 1 & 1 & 1 & 0 & 1
\end{bmatrix}
,$\newline(ii) \newline$B_{1}^{(2)}=%
\begin{bmatrix}
0 & 0 & 1 & 0 & 1 & 1 & 1 & 0 & 0\\
0 & 0 & 1 & 0 & 1 & 0 & 0 & 1 & 1\\
1 & 1 & 0 & 1 & 1 & 1 & 1 & 1 & 1\\
0 & 0 & 1 & 0 & 1 & 1 & 1 & 1 & 1\\
1 & 1 & 1 & 1 & 0 & 1 & 0 & 0 & 1\\
1 & 0 & 1 & 1 & 1 & 0 & 1 & 0 & 1\\
1 & 0 & 1 & 1 & 0 & 1 & 0 & 1 & 1\\
0 & 1 & 1 & 1 & 0 & 0 & 1 & 0 & 0\\
0 & 1 & 1 & 1 & 1 & 1 & 1 & 0 & 0
\end{bmatrix}
,$ ~$A^{(2)}=%
\begin{bmatrix}
0 & 1 & 1 & 0 & 1 & 0 & 0 & 1 & 1\\
0 & 1 & 1 & 0 & 1 & 1 & 0 & 1 & 0\\
0 & 1 & 1 & 0 & 0 & 1 & 0 & 1 & 1\\
1 & 0 & 1 & 0 & 1 & 1 & 1 & 1 & 1\\
1 & 1 & 0 & 0 & 0 & 0 & 0 & 1 & 0\\
1 & 0 & 1 & 1 & 1 & 1 & 0 & 1 & 1\\
1 & 0 & 1 & 1 & 0 & 0 & 1 & 1 & 0\\
1 & 1 & 1 & 0 & 0 & 0 & 0 & 0 & 0\\
0 & 0 & 0 & 1 & 1 & 1 & 1 & 0 & 1
\end{bmatrix}
,$ \newline and \newline(iii) \newline$B_{1}^{(3)}=%
\begin{bmatrix}
0 & 0 & 0 & 0 & 0 & 0 & 1 & 0 & 1\\
0 & 0 & 1 & 0 & 1 & 1 & 0 & 0 & 1\\
0 & 1 & 0 & 0 & 1 & 0 & 0 & 0 & 0\\
0 & 0 & 0 & 0 & 1 & 0 & 0 & 1 & 0\\
0 & 1 & 1 & 1 & 0 & 1 & 0 & 0 & 0\\
0 & 1 & 0 & 0 & 1 & 0 & 0 & 0 & 0\\
1 & 0 & 0 & 0 & 0 & 0 & 0 & 1 & 0\\
0 & 0 & 0 & 1 & 0 & 0 & 1 & 0 & 0\\
1 & 1 & 0 & 0 & 0 & 0 & 0 & 0 & 0
\end{bmatrix}
,$ ~$A^{(3)}=%
\begin{bmatrix}
0 & 0 & 1 & 1 & 1 & 1 & 1 & 0 & 0\\
1 & 0 & 0 & 0 & 1 & 1 & 0 & 1 & 1\\
0 & 1 & 1 & 0 & 1 & 1 & 1 & 0 & 0\\
0 & 1 & 1 & 1 & 1 & 1 & 0 & 1 & 1\\
0 & 1 & 0 & 1 & 0 & 0 & 1 & 1 & 1\\
1 & 0 & 1 & 0 & 0 & 1 & 0 & 1 & 1\\
1 & 0 & 1 & 0 & 1 & 0 & 0 & 0 & 0\\
0 & 0 & 0 & 0 & 0 & 1 & 1 & 1 & 0\\
0 & 0 & 0 & 0 & 0 & 1 & 1 & 0 & 1
\end{bmatrix}
.$\newline
\end{lemma}

Each matrix $B_{1}^{(j)},~j=1,2,3,$ \ has zero diagonal. The generator matrix
(\ref{gmain}) depends on $B_{2}$ and $B_{3}.$ Multiplying the columns of the
matrix (\ref{gmain}) by $g^{2}=1+h^{2}$ as needed we can make the diagonal of
$B_{2}$ to be zero without changing the entries of $B_{1}^{(j)}$ and
$A^{(j)}.$ Condition (iii) of Theorem \ref{Structure4} determines the entries
below the diagonal of $B_{2}.$ Hence, $B_{2}$ depends on 36 parameters.
Condition (iv) of Theorem \ref{Structure4} and Theorem \ref{Th_de} determine
the entries on and below the diagonal of $B_{3}.$ Thus, $B_{3}$ depends on 36
parameters. We obtain the following result.

\begin{corollary}
\label{CorParam} An extremal self-dual code of length 72 with an automorphism
of order 4 is equivalent to one of the $3\left(  2^{72}\right)  $ codes
determined by the matrix (\ref{gmain}), Theorem \ref{Structure4}, Theorem
\ref{Th_de}, and Lemma \ref{B1A}.
\end{corollary}

\section{Further Reduction of the Search Space}

Let $P=\mathbb{F}_{2}[x_{1},x_{2},\ldots,x_{72}]$ be the the polynomial ring
of the indeterminates $x_{1},x_{2},\ldots,x_{72}$ over the binary field
$\mathbb{F}_{2}.$ Corollary \ref{CorParam} shows that the matrices $B_{2}$ and
$B_{3}$ are determined by the entries above the diagonal:
\begin{align*}
B_{2}[1,2]  & =x_{1},B_{2}[1,3]=x_{2},\ldots,B_{2}[1,9]=x_{8},\\
B_{2}[2,3]  & =x_{9},\ldots,B_{2}[2,9]=x_{15},\ldots,B_{2}[8,9]=x_{36},\\
B_{3}[1,2]  & =x_{37},B_{3}[1,3]=x_{38},\ldots,B_{3}[1,9]=x_{44},\\
B_{3}[2,3]  & =x_{45},\ldots,B_{3}[2,9]=x_{51},\ldots,B_{3}[8,9]=x_{72}.
\end{align*}
As the matrices $B_{1}^{(j)}$ and $A^{(j)},~j=1,2,3,$ are determined in Lemma
\ref{B1A}, for any selection of binary values for $x_{1},x_{2},\ldots,x_{72}$
the matrix (\ref{gmain}) determines a doubly even self-dual code $C$ of length
72. The [36,18,8] code $\Phi(C)$ has a generator matrix
\[%
\begin{bmatrix}
I_{2}\otimes I+J_{2}\otimes B_{1}^{(j)} & I_{2}\otimes A^{(j)}%
\end{bmatrix}
\]
where $\otimes$ denotes the Kroneker product, $I_{2}$ and $J_{2}$ are the
identity and the all-one matrices of order 2, correspondingly. Let $G_{j}$ be
the automorphism group of this code, $j=1,2,3.$ We know that $\overline{g}\in
G_{j}.$ \ 

Using Magma we determine the groups $G_{j}$ and $C_{G_{j}}(\overline{g}),$ the
centralizer of $\overline{g}$ in $G_{j}.$ The order of $C_{G_{j}}(\overline
{g})$ for $j=1,2,3$ is 96, 384, and \ 96 and the number of generators is 4, 5,
and 5 , correspondingly. Clearly $C_{G_{j}}(\overline{g})\subseteq C_{S_{36}%
}(\overline{g})$ which is isomorphic to the wreath product $Z_{2}\wr S_{18}.$
On the other hand $C_{S_{72}}(g)$ is isomorphic to $Z_{4}\wr S_{18}.$ Any
permutation from $C_{S_{72}}(g)$ maps $C$ to an equivalent code with
automorphism $g.$

Let $\overline{\tau}$ be a generator of $C_{G_{j}}(\overline{g}).$ We lift
$\overline{\tau}$ to $\tau\in C_{S_{72}}(g)$ having the same permutation part
from $S_{18}$ as $\overline{\tau}$ such that when $\tau$ is applied to the
matrix
\begin{equation}
\lbrack I+B_{1}^{(j)}h+B_{2}h^{2}+B_{3}h^{3}~|~A^{(j)}]\label{C(j)}%
\end{equation}
the matrices $I,$ $B_{1}^{(j)},$ and $A^{(j)}$ do not change and $B_{2}$ maps
to a matrix $B_{2}^{^{\prime}}$ with zero diagonal. The computations show that
the entries above the diagonal of $B_{2}^{^{\prime}}$ are $X\ast T_{\tau
}+v_{\tau},$ where $X=\left(  x_{1},x_{2},\ldots,x_{36}\right)  ,$ $T_{\tau}$
is a 36x36 nonsingular binary matrix and $v_{\tau}$ is a binary vector of
length 36. The affine transformations $\left(  T_{\tau},v_{\tau}\right)  $
generate an affine group $K_{j}$ for $j=1,2,3,$ of order 12288, 49152, 12288,
correspondingly. Every binary code given by (\ref{C(j)}) has a subcode of
dimension 27 defined by the matrix
\begin{equation}
\lbrack Ih+B_{1}^{(j)}h^{2}+B_{2}h^{3}~|~A^{(j)}h]\label{D(j)}%
\end{equation}
We used computations with Magma on a desktop computer to obtain the next lemma.

\begin{lemma}
\label{subcode} The number of orbits of weight 16 codes (\ref{D(j)}) under the
action of the group $K_{j},$ for $j=1,2,3,$ is 501142, 131840, and 925972, correspondingly.
\end{lemma}

\section{The Test of $1558954(2^{36})$ Codes}

The number of orbits from Lemma \ref{subcode} is 1558954. For each orbit we
select a representative and determine the corresponding matrix (\ref{D(j)}).
As the matrix $B_{3}$ from (\ref{C(j)}) depends on 36 binary parameters, the
search space contains $1558954(2^{36})$ codes. In order to speed up the
computations we use a 36-bit binary reflected Gray code \cite{Gray} to order
the vectors of the 36 dimensional binary vector space in a sequence%
\[
u^{(0)},u^{(1)},u^{(2)},\ldots,u^{(n-1)}%
\]
$n=2^{36},$ such that consecutive vectors $u^{(i-1)}$ and $u^{(i)},$
$i=1,~2,\ldots,n-1,$ differ in exactly one bit in position, say $g(i).$ For a
fixed matrix (\ref{D(j)}), the corresponding $2^{36}$ matrices (\ref{C(j)})
form a sequence
\begin{equation}
M^{(0)},M^{(1)},M^{(2)},\ldots M^{(n-1)}\label{Csec}%
\end{equation}
such that $M^{(i)}-M^{(i-1)}=D^{(g(i))}$ is one of 36 predetermined mask
matrices
\[
D^{(1)},D^{(2)},\ldots,D^{(36)}.
\]
Each of the mask matrices has at least 28 zero rows. As a result a low weight
vector found in a code with a generator matrix from (\ref{Csec}) often belongs
to the next several codes.

We carried out these computations on the Janus supercomputer at the University
of Colorado Denver. We wrote a computer code in C programming language that
was highly optimized for both the task and the Janus hardware. The program
employs some of the 64-bit single-cycle bitwise operations of the processors.
We used about 6 million CPU core hours over a period of more than three months
to find a vector of weight less than 16 in each of the $1558954(2^{36})$ codes
from the search space. This completes the proof of the Theorem \ref{main}.

\textbf{Acknowledgement:} This work utilized the Janus supercomputer, which is
supported by the National Science Foundation (award number CNS-0821794) and
the University of Colorado Boulder. The Janus supercomputer is a joint effort
of the University of Colorado Boulder, the University of Colorado Denver and
the National Center for Atmospheric Research. The authors would like to thank
Dr. Jan Mandel, University of Colorado Denver, for his advice in setting up
the computations.

\bibliographystyle{elsart-num}
\bibliography{acompat}

\end{document}